\documentclass{amsart}

\usepackage{enumerate}
\usepackage{amssymb}
\usepackage{amsmath,amscd}
\usepackage{amsthm}
\usepackage{graphicx}
\usepackage{pictexwd,dcpic}

\newcommand{\R}{\mathbb{R}}
\newcommand{\Z}{\mathbb{Z}}
\newcommand{\OO}{\mathrm{O}}
\newcommand{\F}{\mathcal{F}}
\newcommand{\srf}{singular Riemannian foliation}
\newcommand{\In}{\subset}
\newcommand{\codim}{\mathrm{codim}}
\newcommand{\lift}{G}

\newtheorem{theorem}{Theorem}

\newtheorem{corollary}[theorem]{Corollary}
\newtheorem{lemma}[theorem]{Lemma}
\newtheorem{proposition}[theorem]{Proposition}

\newtheorem{maintheorem}{Theorem}

\theoremstyle{definition}
\newtheorem{definition}[theorem]{Definition}

\theoremstyle{remark}
\newtheorem{remark}[theorem]{Remark}

\title{A Slice Theorem for singular Riemannian foliations, with applications}
\date{\today}
\author{Ricardo A. E. Mendes, Marco Radeschi}

\begin{document}

\maketitle
\begin{abstract}
We prove a Slice Theorem around closed leaves in a singular Riemannian foliation, and we use it to study the $C^\infty$-algebra of smooth basic functions, generalizing to the 
inhomogeneous setting a number of results by G.~Schwarz. In particular, in the infinitesimal case we show that this algebra is generated by a finite number of polynomials.
\end{abstract}

\section{Introduction}

We consider certain partitions $\F$ of Riemannian manifolds $M$ called \emph{\srf s}. These are partitions into smooth connected equidistant submanifolds of possibly varying dimension called \emph{leaves} (see Section \ref{S:preliminaries} for the precise definition and some basic facts). 

Singular Riemannian foliations generalize several classes of objects that have been traditionally studied in Riemannian Geometry. One example is the decomposition into orbits under an action of a connected group by isometries, which we will also refer to as a \emph{homogeneous} \srf, studied in the theory of Differentiable Transformation Groups (see for example \cite{Bredon}). Another class is that of Isoparametric Foliations, whose study goes back to Levi-Civita and \'E.~Cartan in the 1930's, but whose origins can be traced back even further to the 1910's. Nevertheless, it remains an important object of current research, see \cite{Thorbergsson} for a survey. Finally, one of course has (regular) Riemannian foliations, which date back to the 1950's \cite{Reinhart59,Haefliger58}. Certain \srf s arise naturally in this context, by taking leaf closures of an arbitrary regular Riemannian foliation, see \cite{Molino}.

An important special case is when the \srf\ is  \emph{infinitesimal}. This means that the ambient Riemannian manifold $M$  is a Euclidean vector space $V$, and the origin is a leaf (cf. Section \ref{S:preliminaries}). Infinitesimal foliations  generalize the orbit decompositions of orthogonal representations, and this is a true generalization, because there exist many \emph{inhomogeneous} infinitesimal foliations. In fact, there is a construction that associates to each infinitesimal foliation an infinite family of (higher dimensional) inhomogeneous infinitesimal foliations. These  are called \emph{composed} foliations, they are related to Clifford systems, and include classical examples such as the octonionic Hopf fibration of $\R^{16}$, and the isoparametric foliations in \cite{FerusKarcherMunzner81} (see \cite{Radeschi14,GorodskiRadeschi15}).  

The celebrated Slice Theorem, first proved by Kozsul \cite{Koszul53}, is a fundamental tool in the Theory of Transformation Groups. Given an action of a Lie group $G$ on a manifold $M$, it describes a tubular neighbourhood around any orbit $G\cdot p$ in terms of a linear representation, namely the \emph{slice representation} of the isotropy group $G_p$ on the normal space to the orbit at $p$. Analogously, given a \srf\ $(M,\F)$ and a point $p\in M$, the normal space at $p$ of the leaf through $p$ carries a natural infinitesimal foliation $\F^p_0$, called the \emph{slice foliation}, cf. Section \ref{S:preliminaries}. Our main result is a Slice Theorem for \srf s:
\begin{maintheorem}[Slice Theorem]
\label{MT:slice}
Let $(M,\F)$ be a singular Riemannian foliation, and let $L$ be a closed leaf with slice foliation $(V,\F^p_0)$ at a point $p\in L$. Then there is a group $K\In \OO(V)$ of foliated isometries of $(V,\F^p_0)$ and a principal $K$-bundle $P$ over $L$, such that for small enough $\epsilon>0$, the $\epsilon$-tube $U$ around $L$ is foliated diffeomorphic to $(P\times_K V^\epsilon,P\times_K\F_0^p)$.
\end{maintheorem}
The group $K$ above consists of all transformations which send each leaf of the disconnected slice foliation $\F^p$ to itself, see section \ref{S:slice} for more details.

A previous version of the Slice Theorem \cite[Theorem 6.1]{Molino}  describes a neighbourhood of a \emph{plaque} (cf. Section \ref{S:preliminaries}). Since a plaque is a neighbourhood of a point inside a leaf, this can be seen as a local version of our  Theorem \ref{MT:slice}. The main ingredient in our proof of Theorem \ref{MT:slice} are linearized vector fields \cite{Molino}, which allow us to reduce the structure of the normal bundle of the leaf to $K$ (see Lemma \ref{L:structure}).

In the second part of this paper we apply Theorem \ref{MT:slice} to the study of smooth \emph{basic functions}, that is, smooth functions on a manifold $M$ that are constant on the leaves of a \srf\ $\F$.

Start with an infinitesimal foliation $(V,\F)$. When the leaves are compact, they are given as the common level sets of a finite number of basic \emph{polynomials}. This is called the Algebraicity Theorem \cite{LytchakRadeschi15}, see Theorem \ref{T:algebraicity}. Generalizing the main result of \cite{Schwarz75} about representations of compact Lie groups, we show:
\begin{maintheorem}
\label{MT:sphere}
Let $\F$ be an infinitesimal \srf\ of the Euclidean space $V$ with compact leaves, and $\rho_1,\ldots,\rho_k$ generators for the algebra of basic polynomials. Then the image of the pull-back map $\rho^*:C^\infty(\R^k)\to C^\infty(V)$ equals the space $C^\infty(V)^\F$ of all smooth basic functions.
\end{maintheorem}

The proof of Theorem \ref{MT:sphere} uses the averaging operator of \cite{LytchakRadeschi15} (cf. Section \ref{S:preliminaries} and Lemma \ref{L:averaging}), together with a result about composite differentiable functions due to Tougeron \cite{Tougeron80} and later generalized by Bierstone-Milman \cite{BierstoneMilman82}.

Next consider a \srf\ $(M,\F)$ of a complete manifold $M$. As an application of Theorems \ref{MT:slice} and \ref{MT:sphere}, we have:
\begin{maintheorem}
\label{MT:manifold}
Let $(M,\F)$ be a \srf\ with closed leaves in the complete manifold $M$. Assume the (disconnected) slice foliations of $M$ fall into a finite number of isomorphism types. Then there are smooth basic functions $f_1, \ldots f_N$
\footnote{See section \ref{S:basicfunctions} for an explicit upper bound for $N$.}
 which generate $C^\infty(M)^\F$ as a $C^\infty$-algebra.
 \end{maintheorem}

Theorem \ref{MT:manifold} was established in the homogeneous case in \cite{Schwarz75}, using the Equivariant Embedding Theorem \cite{Mostow57,Palais57}. In contrast, we use a result from Dimension Theory due to Ostrand \cite{Ostrand65}. We also observe that the literal generalization of the Equivariant Embedding Theorem to inhomogeneous foliations cannot hold, because not every (compact) manifold is a leaf of some \srf\ of some \emph{simply-connected} compact ambient manifold. Indeed, the fundamental group of such a leaf needs to be \emph{virtually Abelian} (Proposition \ref{P:virtuallyabelian}). This is proved as a simple application of \cite[Theorem A]{GGR15}, together with Theorem \ref{MT:slice}.

As another application of Theorem \ref{MT:sphere}, we have:
\begin{maintheorem}[Inverse Function Theorem]
\label{MT:inverse}
Let $(M_i,\F_i)$, $i=1,2$ be singular Riemannian foliations with closed leaves, and let $f:M_1/\F_1\to M_2/\F_2$ be a smooth strata preserving map such that $d_xf$ is an isomorphism for some $x\in M_1/\F_1$. Then $f$ is a diffeomorphism in a neighbourhood of $x$.
\end{maintheorem}

The present article is organized as follows. In section \ref{S:preliminaries} we collect definitions and some basic facts  about \srf s, and fix notations. In section \ref{S:slice} we define disconnected infinitesimal foliations and prove Theorem \ref{MT:slice}, the Slice Theorem. Section \ref{S:basicfunctions} concerns smooth basic functions and is devoted to proving Theorems  \ref{MT:sphere} and \ref{MT:manifold}. In section \ref{S:maps} we study smooth maps between leaf spaces, and prove Theorem \ref{MT:inverse}, the Inverse Function Theorem.




\subsection*{Acknowledgements} We would like to thank Marcos Alexandrino for discussions that led to Theorem \ref{MT:slice}, Alexander Lytchak for several conversations on Theorem \ref{MT:sphere}, and John Harvey for pointing out Ostrand's Theorem to us. We also thank the referees, whose input helped improve the presentation, and simplify the proof of Theorem \ref{MT:sphere}.

\section{Preliminaries}
\label{S:preliminaries}

We collect in this section the definitions and basic facts that will be used in the rest of this paper.

\begin{definition}
Let $M$ be a smooth manifold.
\begin{itemize}
\item
 A \emph{generalized distribution} is an assignment $x\mapsto \mathcal{D}_x$ of a linear subspace $\mathcal{D}_x\subset T_xM$ to each $x\in M$. It is \emph{smooth} if  every $X_0\in \mathcal{D}_x$ extends to a smooth vector field $X$ on $M$, such that $X(y)\in \mathcal{D}_y $ for every $y\in M$.
\item
A smooth generalized distribution $\mathcal{D}$ is \emph{integrable} if, for every $x\in M$, there is a connected, smooth, immersed submanifold $L\subset M$  containing $x$, such that $T_yL=\mathcal{D}_y$ for all $y\in L$. Such a submanifold $L$ is called an \emph{integral manifold} of $\mathcal{D}$.
\item
A \emph{singular foliation} is the partition of $M$ into the maximal integral manifolds (called \emph{leaves}) of a smooth integrable distribution.
\end{itemize}
\end{definition}

The celebrated Stefan-Sussmann Theorem \cite{Stefan74,Sussmann73} gives a sufficient and necessary condition for a smooth generalized distribution to be integrable.

Following \cite[page 189]{Molino}, we define the main objects of the present paper by combining singular foliations  with the notion of a transnormal system due to Bolton \cite{Bolton73}:
\begin{definition}
A \emph{\srf} of a Riemannian manifold $M$ is a singular foliation $\F$ such that every geodesic that is normal to a leaf is normal to every other leaf it meets.
\end{definition}
We will also use the notation $(M,\F)$ when we want to emphasize the ambient manifold $M$. The tangent space to the leaf $L_p$ through $p$ will be called the \emph{vertical} space at $p$, and its orthogonal complement the \emph{horizontal} space. Vector fields tangent to all leaves will be called \emph{vertical} vector fields.
The orbit decomposition under the isometric action of a connected Lie group forms a \srf, and these will be referred to as \emph{homogeneous} \srf s.

Any two leaves $L_1, L_2$ of a \srf\ $\F$ are \emph{equidistant} in the sense that $d(L_1,q)$ is independent of $q\in L_2$. In fact, a singular foliation is Riemannian if and only if leaves are \emph{locally} equidistant . To make this precise and also to state Molino's Homothetic Transformation Lemma, we need the following definitions from \cite[page 192]{Molino}:
\begin{definition}
Let $M$ be a Riemannian manifold, and $\F$ a singular foliation of $M$. A \emph{plaque} is an open subset $P\subset L$ of a leaf $L$, such that the inclusion of $P$ in $M$ is an embedding. Denote by $\nu P$ the normal bundle of $P$, and by $\exp:\nu P\to M$ the normal exponential map. If $\epsilon>0$ is such that $\exp$ restricted to $\{v\in \nu P\ |\ \|v\|<\epsilon\}$ is a diffeomorphism onto its image $O_\epsilon(P)$, then $O_\epsilon(P)$ is called a \emph{distinguished tubular neighbourhood} of $P$. Denote by $\F_O$ the partition of $O_\epsilon$ into the connected components of the intersections of $O_\epsilon$ with the leaves of $\F$.
\end{definition}
In the notations above, a singular foliation $\F$ is Riemannian if and only if every $p\in M$ is contained in a distinguished tubular neighbourhood $O_\epsilon$, such that the leaves of $\F_O$ are equidistant (see \cite[page 193]{Molino}). 

\begin{lemma}[Homothetic Transformation Lemma \cite{Molino}, page 193]
\label{L:homothetic}
 Let $\F$ be a \srf\ of a Riemannian manifold $M$, and fix a distinguished tubular neighbourhood $O_\epsilon$. Then, for any $\lambda>1$, the homothetic transformation $h_\lambda:O_{\epsilon/\lambda}\to O_\epsilon$ given by $\exp(x)\mapsto \exp(\lambda x)$ takes leaves of $\F_O$ to leaves of $\F_O$. 
\end{lemma}

We will need the following fact, see Proposition 4.3 in \cite{LytchakThorbergsson10}, or Theorem 2.9 in \cite{Alexandrino10} for a proof.
\begin{proposition}[Equifocality]
\label{P:equifocality}
Let $\F$ be a \srf\ of a Riemannian manifold $M$, and let $L$ be a leaf. If $v,w\in\nu L$ are two normal vectors (at possibly different points) such that $\exp(tv)$ and $\exp(tw)$ belong to the same leaf for all small $t>0$, then they belong to the same leaf for all $t\in \R$ for which $\exp(tv)$ and $\exp(tw)$ are defined.
\end{proposition}

\begin{definition}
An \emph{infinitesimal foliation} is a \srf\ of a Euclidean vector space $V$, such that the origin is a leaf.
\end{definition}
Given an infinitesimal foliation $(V,\F)$, the unit sphere $S(V)$ centered at the origin is a union of leaves (by equidistance from the origin), and the resulting partition $\F|_{S(V)}$ of $S(V)$ is in fact a \srf. Moreover, by the Homothetic Transformation Lemma,  $\F|_{S(V)}$ completely determines $\F$. This construction defines a  one-one correspondence between infinitesimal foliations and \srf s of (round) spheres.

Infinitesimal foliations appear naturally: Given a \srf\  $(M,\F)$, let $L$ be the leaf through $p\in M$, and define the \emph{slice} at $p$ by $V=\nu_pL$, with its natural Euclidean metric. For small $\epsilon>0$, let $V^\epsilon=\{v\in V\ |\ \|v\|<\epsilon\}$, and consider the partition $\F^p$ of $V^\epsilon$ defined by intersecting the leaves of $\F$ with $\exp_p(V^\epsilon)\simeq V^\epsilon$. Define the (connected) \emph{slice foliation} $\F^p_0$ by taking connected components of  $\F^p$,  and then extending this foliation by homotheties to all of $V$. This is well-defined by the Homothetic Transformation Lemma, and $(V,\F^p_0)$ is an infinitesimal foliation by \cite[Thm 6.1 and Prop. 6.5]{Molino}. When $\F$ is homogeneous, the slice foliation $(V,\F^p_0)$ coincides with the (orbit decomposition under the) slice representation of the \emph{connected component} of the isotropy group at $p$.

A recently established \cite{LytchakRadeschi15} fundamental property of infinitesimal foliations with compact leaves is their \emph{algebraic} nature. More precisely:
\begin{theorem}[Algebraicity]
\label{T:algebraicity}
Let $(V,\F)$ be an infinitesimal foliation with compact leaves. Denote by $\R[V]^\F$ the algebra of \emph{basic polynomials}, that is, polynomials on $V$ that are constant on  the leaves of $\F$.  Then $\R[V]^\F$ is finitely generated and separates leaves. In particular, given a set of generators $\rho_1,\ldots \rho_k$, any leaf is the inverse image of a point under the map $\rho=(\rho_1,\ldots, \rho_k):V\to\R^k$. Moreover, the map $\rho$ induces a homeomorphism between the leaf space $V/\F$ and a semi-algebraic subset of $\R^k$.
\end{theorem}
The Algebraicity Theorem reduces to Hilbert's Theorem in the homogeneous case, and to \cite[Satz 2]{Muenzner80} and \cite[Theorem D]{Terng85} in the isoparametric case (see also \cite{Thorbergsson91}).

Key in the proof of the Algebraicity Theorem is the \emph{averaging operator} $\mathrm{Av}$. Given $f\in C^\infty(V)$, one defines the value of $\mathrm{Av}(f):V\to \R$ at $p\in V$ as the average of $f|_{L_p}$ with respect to the standard volume form of $L_p$. The main property of the averaging operator is that $\mathrm{Av}(f)$ is actually smooth, so that in particular $\mathrm{Av}$ takes polynomials to polynomials. 

\section{The Slice Theorem}
\label{S:slice}

The goal of this section is to prove Theorem \ref{MT:slice}, the Slice Theorem, that will in turn be used in the remainder of this article.

In subsection \ref{SS:disconnected} we give the definition of (possibly) \emph{disconnected infinitesimal foliations}\footnote{The definition here is slightly different from the one given in \cite{AlexandrinoRadeschi15}.}, which is a generalization of representations by (possibly) disconnected groups, and of infinitesimal \srf s. This notion  is important in the proof of Theorem \ref{MT:slice} because, for a leaf $L$ with slice $V$,  the intersections of nearby leaves with $V$ naturally form a disconnected infinitesimal foliation in this sense, which moreover may fail to have connected leaves, even though the leaves of $\F$ are connected (as is already the case for group actions).  We show that the Algebraicity Theorem \cite[Theorem 1.1]{LytchakRadeschi15} carries over to disconnected infinitesimal foliations. We also show that Theorem \ref{MT:sphere} holds for  disconnected infinitesimal foliations, provided it holds for infinitesimal foliations (with connected leaves). 

We continue in subsection \ref{SS:linearization} with a definition and properties of linearized vector fields. The definition is somewhat different but equivalent to the one given in \cite{Molino}, section 6.4. Basic properties of linearized vector fields together with equifocality (see Proposition \ref{P:equifocality}, or  \cite[Proposition 4.3]{LytchakThorbergsson10}, or  \cite[Theorem 2.9]{Alexandrino10}) allow one (Corollary \ref{C:equifocality}) to ``push'' or ``slide'' a normal geodesic along a path in a leaf without changing its image in the leaf space.  This is a basic tool that can be used, in particular, to prove  (Proposition \ref{P:tube}) that closed leaves admit a tubular neighbourhood of constant radius, a fact that is well-known by experts, but for which the authors could not find a detailed proof. In case all leaves of a \srf\ $\F$ are closed, Proposition \ref{P:tube} can be interpreted as follows: every point $x$ in the leaf space $M/\F$ has a neighbourhood where $d(x,\cdot)^2$ is smooth. This in particular yields  smooth partitions of unity, see subsection \ref{SS:manifolds}.

Finally, in subsection \ref{SS:localmodels} we prove Theorem \ref{MT:slice} (the Slice Theorem), which describes a small tubular neighbourhood of a closed leaf up to foliated diffeomorphism. This is accomplished by reducing the structure group of the normal bundle of the leaf with the help of linearized vector fields. We also present a converse to the Slice Theorem.

\subsection{Disconnected infinitesimal foliations}
\label{SS:disconnected}

We define a generalization of infinitesimal \srf s where the leaves are allowed to be disconnected. 

Let $(V,\F_0)$ be an infinitesimal foliation, that is, a singular Riemannian foliation on a Euclidean space $V$ with $\{0\}$ a leaf of $\F_0$ (cf. Section \ref{S:preliminaries}). Let $\OO(V,\F_0)$ denote the group of isometries of $V$ sending leaves of $\F_0$ to (possibly different) leaves of $\F_0$, and $\OO(\F_0)\In \OO(V,\F_0)$ denote the normal subgroup of isometries sending each leaf to itself. Note that $\OO(V,\F_0)/\OO(\F_0)$ is the group of bijections of the leaf space $V/\F_0$ which lift to maps in $\OO(V)$. 

\begin{definition}[Disconnected foliation]
\label{D:disco}
A \emph{disconnected infinitesimal foliation} is a partition $\F$ of a Euclidean vector space $V$ into smooth submanifolds, satisfying two properties. First, the partition $\F_0$ obtained from $\F$ by taking path-connected components of leaves is an infinitesimal \srf; and second, there is a subgroup $\Gamma\In\OO(V,\F_0)/\OO(\F_0)$ whose action on the leaf space $V/\F_0$ is transitive on the connected components of each leaf of $\F$.
\end{definition}
When the leaves of $(V,\F_0)$ are compact, $V/\F_0$ is a Hausdorff metric space, and $\Gamma$ is a subgroup of its isometry group. When moreover the leaves of $\F$ are compact, $\Gamma$ has finite orbits, hence must be finite. In this case, the finite group $\Gamma$ is uniquely determined by $\F$.
 
Note also that the leaves of a disconnected infinitesimal foliation $(V,\F)$ are contained in spheres centered at the origin, and that the conclusion of the  Homothetic Transformation Lemma (see Lemma \ref{L:homothetic}) holds. In particular $(V,\F)$ is completely determined by its restriction to the unit sphere $S\In V$.

We mention two natural  examples of disconnected infinitesimal foliation. First, \emph{homogeneous} disconnected infinitesimal foliations are given by $K$-orbits, where $K\In\OO(V)$ is an arbitrary Lie subgroup. Here $\F_0$ is given by the orbits of the path-connected component $K_0$ of $K$. The group $K/K_0$ acts on $V/\F_0$, and $\Gamma$ is isomorphic to the quotient of $K/K_0$ by the innefective kernel. 
Second, if $\F$ is a \srf\ in a complete manifold $M$ and $L$ is a closed leaf, the partition $\F^p$ of the slice $V_p$ at a point $p\in L$ into the intersections of leaves of $\F$ with $V_p$ constitutes a disconnected infinitesimal foliation. Here $\pi_1(L)$ acts on $V/\F_0^p$ and $\Gamma$ is isomorphic to the quotient of $\pi_1(L)$ by the innefective kernel, see subsection \ref{SS:localmodels}.

\begin{remark} 
On the other hand, there are naturally occuring ``foliations'' with disconnected leaves that are \emph{not} disconnected infinitesimal foliations in the sense of Definition \ref{D:disco}. One source of such ``non-examples'' are infinitesimal  \srf s $\F_0$ with closed leaves, whose leaf space $V/\F_0$ admits  isometries that do not lift to $\OO(V)$. If $\Gamma$ is a (say finite) group containing such non-liftable isometries, then one may define the leaves of $\F$ to be the inverse images of $\Gamma$-orbits under the quotient map $V\to V/\F_0$. Some concrete homogeneous examples of such $\F_0$ may be extracted from the classification of orbit spaces of cohomogeneity $3$ representations, see \cite[Lemma 6.1]{Straume94}. They include polar representations with generalized Weyl group of the form $\Z/2\times D_4$ or $\Z/2\times D_6$, where $D_k$ denotes the dihedral group with $2k$ elements. Other such $\F_0$ include the quaternionic Hopf foliation of $\R^8$, the (inhomogeneous) octonionic Hopf foliation of $\R^{16}$, and, more generally, the Clifford foliation associated to any Clifford system $C_{m,k}$ with $m\equiv 0 \ (\text{mod }4)$ and $k$ odd, see \cite{Radeschi14}.

A different class of non-examples are the level sets in round spheres of some homogeneous polynomial $F$ such that $\|\nabla F\|^2(x)$ is a power of $\|x\|^2$.
These were classified in \cite{Tkachev14} and define a partition $\F$ into equidistant possibly disconnected submanifolds, such that $\F_0$ is isoparametric with codimension one regular leaves.
\end{remark}

Next we extend the Algebraicity Theorem (see Theorem \ref{T:algebraicity}, or \cite[Theorem 1.1]{LytchakRadeschi15}) to  disconnected infinitesimal foliations:
\begin{lemma}[Algebraicity]
\label{L:algebraicity}
Let $(V,\F)$ be a disconnected infinitesimal foliation with compact leaves. Then there are homogeneous polynomials $\psi_1, \ldots, \psi_l\in\R[V]$ such that the leaves of $\F$ are precisely the pre-images of points under $\psi=(\psi_1, \ldots, \psi_l):V\to \R^l$. Moreover $\psi$ induces a homeomorphism between $V/\F$ and the semi-algebraic subset $\psi(V)\In \R^l$.
\end{lemma}
\begin{proof}
Let $\rho_1,\ldots \rho_k$ be homogeneous generators for the algebra $A$ of basic polynomials of $\F_0$. By \cite[Theorem 1.1]{LytchakRadeschi15}, the map $\rho=(\rho_1,\ldots \rho_k):V\to \R^k$ induces a homeomorphism from $V/\F_0$ to $\rho(V)\In \R^k$.
We claim that $\rho_1,\ldots \rho_k$ may be chosen so that the induced action of $\Gamma$ on $\rho(V)=V/\F_0$ extends to a linear action of $\Gamma$ on $\R^k$.

Indeed, the action of $\Gamma$ on $V/\F_0$ defines a natural action of $\Gamma$ on $A$ by graded algebra automorphisms. This action  preserves the natural inner product  on each graded part  $A_i\In\mathrm{Sym}^iV$ induced by the one on $V$. Denoting by $I$ the ideal in $A$ generated by all homogeneous elements of positive degree, we take  $\rho_1,\ldots \rho_k$ to be any set of homogeneous polynomials which form a basis for the orthogonal complement complement of $I^2$ inside $I$. Thus the action of $\Gamma$ on $A$ induces a linear action on $\R^k=\rm{span}( \rho_1,\ldots \rho_k)$, which extends the given action on $\rho(V)=V/\F_0$.

By Hilbert's Theorem, the algebra of polynomials on $\R^k$ invariant under the $\Gamma$-action is finitely generated by say $\eta_1, \ldots, \eta_m$, and separates orbits. Therefore the polynomials $\eta_i\circ\rho$ generate the subalgebra $B\In A$ consisting of basic polynomials for $\F$, and separate leaves. 

Finally, $B$ is a graded algebra by the Homothetic Transformation Lemma. Thus $B$ is generated by a finite number of homogeneous polynomials, and any such generating set makes up the desired $\psi_1, \ldots \psi_l$.
\end{proof}

The lemma below says that if the conclusion of Theorem \ref{MT:sphere} holds for $(V,\F_0)$, then it also holds for $(V,\F)$, where $\F$ is a disconnected infinitesimal foliation.
\begin{lemma}
\label{L:discosmooth}
With the same notations as in Lemma \ref{L:algebraicity}, $$C^\infty(V)^{\F_0}=\rho^*(C^\infty(\R^k))\  
\Longrightarrow \ C^\infty(V)^\F=\psi^*(C^\infty(\R^l)).$$
\end{lemma}
\begin{proof}
Let $f\in C^\infty(V)^\F$, that is, a smooth function constant on the leaves of $\F$. Since the leaves of $\F$ are unions of leaves of $\F_0$, the hypothesis says that there is $g\in C^\infty(\R^k)$ such that $f=g\circ \rho$. Then $g|_{\rho(V)}$ is $\Gamma$-invariant, so that defining $\tilde{g}$ to be the result of averaging $g$ over $\Gamma$, we have $f=\tilde{g}\circ \rho$.

By Schwarz's Theorem \cite{Schwarz75}, there is $h\in C^\infty(\R^m)$ such that $\tilde{g}=h\circ\eta$, and therefore $f=h\circ\eta\circ\rho$. Finally, each $\eta_i\circ\rho$ is a polynomial in  $\psi_j$, and thus $f$ is a smooth function of $\psi_j$.
\end{proof}

\subsection{Linearized vector fields}
\label{SS:linearization}

Let $M$ be a complete Riemannian manifold, let $L\In M$ be a closed submanifold, and let $X$ be a vector field on $M$ which is tangent to $L$. Denoting by $\nu L$ the normal bundle of $L$, there exists a neighbourhood $W$ of the zero section in $\nu L$ and a neighbourhood $U$ of $L$ in $M$ such that $\exp^{\perp}|_W:W\to U$ is a diffeomorphism. The preimage $(\exp^\perp)^{-1}_*X|_U$, which we also denote by $X$, is a smooth vector field on $W$.

\begin{definition}
Given a vector field $X$ on $W\In\nu L$, we define the \emph{linearization of $X$ around $L$}  as
\[
X^\ell=\lim_{\lambda\to 0} (r_\lambda^{-1})_*(X\circ r_{\lambda}).
\]
Where $r_{\lambda}:\nu L\to \nu L$ is the rescaling $(p,v)\to (p,\lambda v)$. We say $X$ is \emph{linearized} if it coincides with its linearization.
\end{definition}

\begin{proposition}[Linearization]\label{P:linearization}
Let $X$ be a smooth vector field on $W$. Then its linearization $X^\ell$ is a well-defined, smooth vector field defined on the whole of $\nu L$, 
 invariant under rescalings. In particular, $X^\ell$ is basic for the foot point projection $r_{0}:\nu L\to L$.
\end{proposition}

\begin{proof}
The proposition is local in nature, so we can restrict our attention to a trivializing open set $B$, with trivialization $\phi: B\times \R^k\to \nu L|_B$, where $k=\codim\, L$.

Fix  coordinates $(x_1, \ldots x_m)$ on $B$ and let $(y_1, \ldots y_k)$ be the standard coordinates of $\R^k$.  The vector field $X$ can be written as
\[
X=\sum_{i=1}^m a_i(x,y){\partial\over \partial x_i}+\sum_{i=1}^kb_i(x,y){\partial \over \partial y_i}.
\]
Since $X|_L$ is tangent to $L$, it follows that $b_i(x,0)=0$ for every $i=1,\ldots k$ and every $x\in B$. It is easy to check that
\[
(r_\lambda^{-1})_*(X\circ r_{\lambda})=\sum_{i=1}^m a_i(x,\lambda y){\partial\over \partial x_i}+\sum_{i=1}^k \frac{b_i(x,\lambda y)}{\lambda} {\partial \over \partial y_i}.
\]
Taking the limit,
\begin{equation}
\label{E:linearization}
X^\ell=\sum_{i=1}^m a_i(x,0){\partial\over \partial x_i}+\sum_{i=1}^k\left(\sum_{j=1}^kC_{ij}(x)y_j\right){\partial \over \partial y_i}.
\end{equation}
where $C_{ij}(x)={\partial b_i\over \partial y_j}(x,0)$. From this formula it is clear that $X^\ell$ is well-defined and smooth in $W$, and invariant under homothetic transformations. In particular, there exists a unique extension of $X^\ell$ to the whole of $\nu L$. In fact, given $(p,v)$ in $\nu L$, define $X^\ell_{(p,v)}=(r_\lambda^{-1})_*X^\ell_{(p,\lambda v)}$ where $\lambda$ is small enough that $\lambda v\in W$.
\end{proof}

%

Now we specialize the discussion above to the case where the closed submanifold $L\In M$ is a leaf of a \srf\   $\F$. Recall that using the Homothetic Transformation Lemma (see Lemma \ref{L:homothetic}, or \cite[Lemma 6.2]{Molino}), there is a unique singular foliation on the normal bundle $\nu L$ that is scaling invariant and which corresponds to $\F|_U$ via the normal exponential map, where $U$ is any small enough tubular neighbourhood around $L$ (see also \cite{LytchakThorbergsson10}, section 4.3). 

\begin{proposition} [Linear flows]
\label{P:flow} Let $(M,\F)$ be a \srf\ , $L$ a closed leaf, and $X$ a vector field tangent to the leaves.  Then the linearization $X^\ell$ around $L$ and its flow $\Phi^t:\nu L\to \nu L$, $t\in \R$, satisfy:
\begin{enumerate}[a)]
\item $X^\ell$ is tangent to the leaves of $(\nu L,\nu \F)$, in particular $\Phi^t$ preserves leaves. 
\item For any $p\in L$, the restriction $\Phi^t|_{\nu_p L}$ is a linear orthogonal transformation from $\nu_p L$ onto $\nu_{\Phi^t(p)}L$.
\end{enumerate}
\end{proposition}
\begin{proof}
a) Since $X$ is tangent to the leaves, so is $(r_\lambda^{-1})_*X\circ r_\lambda$ by the Homothetic Transformation Lemma (Lemma \ref{L:homothetic}). By taking the limit $\lambda\to 0$, $X^\ell$ is tangent to the leaves.

b) By Proposition \ref{P:linearization}, $X^\ell$ is basic with respect to the footpoint projection, and therefore $\Phi^t$ takes fibers of $\nu L\to L$ to fibers. By Equation \eqref{E:linearization} the restriction $\Phi^t|_{\nu_p L}$ is a linear map. Since $\F$ is a \srf , the leaves are contained in distance tubes around $L$. By the previous item, $\Phi^t$ preserves leaves, so that the restriction $\Phi^t|_{\nu_p L}$ preserves the norm. In other words, it is an orthogonal transformation.
\end{proof}

\begin{corollary}
\label{C:equifocality}
Let $L$ be a closed leaf of a \srf\ $(M,\F)$, and $\gamma:[0,1]\to L$ a piecewise smooth curve with $\gamma(0)=p$. Then there exists a continuous map $\lift:[0,1]\times \nu_p L\to \nu L$ such that
\begin{enumerate}[a)]
\item $\lift(t,v)\in \nu _{\gamma(t)}L$ for every $(t,v)\in [0,1]\times \nu_p L $.
\item For every $t\in[0,1]$, the restriction $\lift|_{\{t\}\times\nu_p L}:\nu_p L\to \nu_{\gamma(t)} L$ is a linear isometry preserving the leaves of $\nu L$.
\item For every $s\in \R$, $\exp_{\gamma(t)}s\lift(t,v)$ belongs to the same leaf as $\exp_psv$.
\end{enumerate}
\end{corollary}
\begin{proof}
Let $0=t_0<t_1<\ldots < t_N=1$ be a partition such that for each $i=1, \ldots N$,  $\gamma_i=\gamma|_{[t_{i-1},t_i]}$ is an embedding, and thus the integral curve of some smooth vector field $X_i$ on $L$. Since $\F$ is a singular foliation, $X_i$ may be extended to a vector field on $M$ that is tangent to all leaves, which we again call $X_i$. Let $X_i^\ell$ denote the linearization of $X_i$ around $L$, and $\Phi_i^t$ its flow. 
For $(t,v)\in [t_{i-1},t_i]\times \nu_pL$, define
$$\lift(t,v)=\Phi_i^{t-t_{i-1}}\circ \Phi_{i-1}^{t_{i-1}-t_{i-2}} \circ \cdots \circ \Phi_2^{t_2-t_1}\circ\Phi_1^{t_1} (v)  $$
Parts a) and b) follow from Proposition \ref{P:flow}.

Since $v$ and $\lift(t,v)$ belong to the same leaf in $\nu L$, part c)  follows directly from equifocality (see Proposition \ref{P:equifocality}, or Proposition 4.3 in \cite{LytchakThorbergsson10}, or Theorem 2.9 in \cite{Alexandrino10}). 
\end{proof}

As an application we can recover the following result, which is well-known to the experts:
\begin{proposition}[Tube with constant radius]
\label{P:tube}
Let $(M,\F)$ be a \srf\ with $M$ complete, and let $L$ be a closed leaf. Then there is $\epsilon>0$ such that the normal exponential map $\exp :\nu^\epsilon L\to M$ is a diffeomorphism onto its image.
\end{proposition}
\begin{proof}
Fix any point $p\in L$ and let $\epsilon>0$ such that, for every $x\in\nu^\epsilon_p L$, the geodesic segment  $\exp_p(tx)$, $0\leq t\leq 1$, is the unique path of minimal length between $L$ and $\exp_p(x)$ in $M$.

It is enough to show that $\exp :\nu^\epsilon L\to M$ is injective. Suppose not. Then there are distinct horizontal geodesic segments $c_1, c_2$ of lengths $\ell_1\leq \ell_2<\epsilon$ joining points $q_1, q_2\in L$ to the same point $r$. Parametrize these segments by arc length so that $c_1(0)=c_2(0)=r$.


Let $\gamma:[0,1]\to L$ be a piecewise smooth curve joining $q_2$ to $p$, and take a lift $\lift:[0,1]\times \nu_{q_2}L\to \nu L$ as in Corollary \ref{C:equifocality}. Let $v\in\nu_p L$ be given by $v=\lift(1,-c_2'(\ell_2))$, so that $\exp_p sv$ belongs to the same leaf as $c_2(\ell_2-s)$ for every $s$. In particular for $s=\ell_2$, the point $r'=\exp_p \ell_2 v$ belongs to $L_r$. Analogously, taking a curve in $L_r$ from $r$ to $r'$, we can move $c_1'(0),c_2'(0)$ to \emph{distinct} unit normal vectors $w_1, w_2\in \nu_{r'}L_r$ such that $\exp_{r'}\ell_i w_i$ belong to $L$, $i=1,2$.

Therefore, by the choice of $\epsilon$, and $l_1\leq l_2$, the three paths $\exp_psv$, $\exp_{r'}sw_1$ and $\exp_{r'}sw_2$, for $0\leq s\leq 1$, must coincide. But this contradicts the fact that $w_1\neq w_2$.
\end{proof}

\begin{figure}[!htb]
\label{figure}
\caption{Proof of Proposition \ref{P:tube}}
\includegraphics[width=0.7\textwidth]{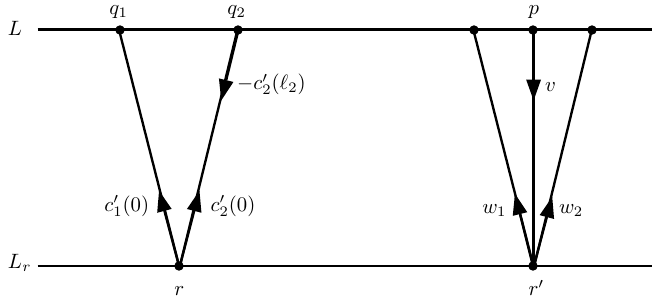}
\end{figure}

\subsection{Local models and the Slice Theorem}
\label{SS:localmodels}

Let $(M,g)$ be a complete Riemannian manifold, and $\F$ a \srf\ of $M$. Fix a closed leaf $L$ and a point $p\in L$. Denote by $\nu L$ the normal bundle. There is a unique singular foliation on $\nu L$ which is scaling invariant and corresponds to $\F|_U$ via the normal exponential map, where $U$ is a small enough neighbourhood of $L$.

Define the \emph{slice} $V=V_p$ at $p$ by $V=\nu_pL$, and the \emph{disconnected slice foliation} $\F^p$ of $V$ by first intersecting the leaves of $\F|_U$ with $\exp_p(V^\epsilon)\simeq V^\epsilon$, and then extending this foliation by homotheties on the whole of $V$.
Define the (connected) \emph{slice foliation} $\F^p_0$ by taking connected components of the leaves of $\F^p$.
We claim that $\F^p$ is a disconnected infinitesimal foliation in the sense of subsection \ref{SS:disconnected}. Indeed, $\F^p_0$ is an infinitesimal \srf\  by \cite[Thm 6.1 and Prop. 6.5]{Molino}. Moreover, there is a homomorphism $\pi_1(L)\to \OO(V,\F^p_0)/\OO(\F^p_0)$ sending a loop $\gamma$ to a lift $\lift(1,\cdot)\in \OO(V,\F^p_0)$ of $\gamma$ in the sense of Corollary \ref{C:equifocality}. By construction, the image $\Gamma$ of this homomorphism acts transitively on the connected components of each leaf of $\F^p$. Indeed, two points $q_1,q_2$ in the same leaf of $\F^p$ are, by definition, contained in a common leaf $L'$ of $\F$. Since $L'$ is connected, one can join $q_1$ and $q_2$ via a path in $L'$, whose projection to $L$ is the desired loop $\gamma$.



The main step in  our  proof of the Slice Theorem is the following lemma, which in turn is proved using linearized vector fields (see subsection \ref{SS:linearization}).
\begin{lemma}
\label{L:structure}
With notation above,  the structure group of the  normal bundle $\nu L$ reduces to $\OO(\F^p)$.
\end{lemma}
\begin{proof}
We construct explicit local trivializations.
Let $q\in L$ and $B$ a small ball in $L$ centered at $q$. Fix a coordinate system $x_1, \ldots x_n$ on $B$ such that $(x_1,\ldots,x_n)(q)=(0,\ldots,0)$, and a piecewise smooth curve $\gamma$ joining $p$ to $q$. By Corollary \ref{C:equifocality}, there is $\lift:[0,1]\times V\to \nu L$ such that $\lift(1,\cdot)$ is an isometry $V\to\nu_q L$ that preserves the leaves of $\F$. In a way similar to the proof of Corollary \ref{C:equifocality}, we  construct 
$$\lift^{(x_1,\ldots, x_n)}:\nu_q L\to\nu_{(x_1,\ldots, x_n)} L$$
 a linear isometry preserving leaves of $\F$ that depends smoothly on $(x_1,\ldots,x_n)$. Namely, we choose linearized vertical vector fields $X_i$ extending $\frac{\partial}{\partial x_i}$, denote their flows by $\Phi_i$, and let 
$$\lift^{(x_1,\ldots, x_n)}=\Phi_n^{x_n}\circ\cdots\circ \Phi_1^{x_1}$$
Then $\lift^{(x_1,\ldots, x_n)}\circ\lift(1,\cdot)$ defines a trivialization of $\nu L$ over $B$, and repeating the same procedure for a collection of small balls covering $L$ yields the desired  reduction of structure group.
\end{proof}

We are now ready to prove the Slice Theorem:
\begin{proof}[Proof of Theorem \ref{MT:slice}]
By Proposition \ref{P:tube}, there is $\epsilon>0$ such that the normal exponential map is a diffeomorphism from $\nu^ \epsilon L$ to an open tube $U\In M$. Take $P$ the principal $\OO(\F^p)$-bundle associated to the local trivializations given by Lemma \ref{L:structure}.

Consider the singular foliation $P\times\F_0^p$ of $P\times V$ whose leaves are $P\times L'$, where $L'$ runs through the leaves of $\F^p_0$. Since $\OO(\F^p)$ preserves the leaves of $\F^p_0$, there is a well-defined induced foliation $\bar{\F}=P\times_{\OO(\F^p)}\F_0$ of the quotient $P\times_{\OO(\F^p)} V$.

By construction of the local trivializations, it is clear that the image of a leaf of $\bar{\F}$ under the diffeomorphism $P\times_{\OO(\F^p)} V^\epsilon\to U$ is contained in one leaf of $\F|_U$. On the other hand, since the leaves of $\F|_U$ are connected, such image must coincide with a leaf.
\end{proof}

In particular, we have the following description of smooth basic functions on a tube around a closed leaf:
\begin{corollary}
\label{C:slice}
Let $(M,\F)$ be a \srf\ of the complete manifold $M$,   and $L$ a closed leaf through $p\in M$. Then the  restriction to the slice $\exp_p(V^\epsilon)\simeq V^\epsilon$ 
$$ |_{V^\epsilon}:C^\infty(U)^{\F|_U}\to C^\infty(V^\epsilon)^{\F^p}$$
is an isomorphism between the spaces of smooth basic functions, where $U$ is an $\epsilon$-tube around $L$, and $\F^p$ is the slice foliation.
\end{corollary}
\begin{proof}
The restriction map is injective because the slice $V^\epsilon$ meets all leaves of $\F|_U$. This in turn follows from the existence of smooth vertical vector fields which generate $\F$, as in corollary \ref{C:equifocality}. 

Turning to surjectivity, let $f\in C^\infty (V^\epsilon)^{\F^p}$. We construct a smooth basic extension of $f$ to $U$. Start with the function $P\times V^\epsilon\to\R$ given by $(x,y)\mapsto f(y)$. It is $\OO(\F^p)$-invariant, and hence defines a smooth basic function $\bar{f}:P\times_{\OO(\F^p)} V^\epsilon\to\R$. By Theorem \ref{MT:slice}, $\bar{f}$ corresponds to a smooth basic function on $U$, which restricts to $f$. 
\end{proof}


Finally, we present a converse to Theorem \ref{MT:slice} for the sake of completeness. Consider a 4-tuple $(L,V,\F,P)$ such that $L$ is a connected manifold, $(V,\F)$ a disconnected infinitesimal foliation, and $P$ is a principal $\OO(\F)$ bundle over $L$, subject to the following condition:  letting $H\In \pi_0(\OO(\F))$ denote the image of the map $\alpha:\pi_1(L)\to \pi_0(\OO(\F))$ induced by $P$, then $H$ acts transitively on $\pi_0(L')$ for every leaf $L'$ of $\F$.

Given any such tuple one can define $(P\times_{\OO(\F)}V, P\times_{\OO(\F)}\F_0)$ as in the proof of Theorem \ref{MT:slice}, and the condition above ensures that the leaves of $P\times_{\OO(\F)}\F_0$ are connected.

\begin{proposition}[Converse of Slice Theorem]\label{P:converse}
Given any 4-tuple $(L,V,\F,P)$ as above, there exists a metric on $P\times_{\OO(\F)}V$ such that $(P\times_{\OO(\F)}V, P\times_{\OO(\F)}\F_0)$ becomes a singular Riemannian foliation, and the metric projection onto $L=P\times_{\OO(\F)}\{0\}$ is a Riemannian submersion.
\end{proposition}

\begin{proof}
The inclusion $\OO(\F)\In \OO(V)$ induces an inclusion $P\In P'$ for some principal $\OO(V)$ bundle over $L$. The bundle $P'$ admits an $\OO(V)$-invariant metric $g_{P'}$, and the restriction $g_P$ of $g_{P'}$ to $P$ is then $\OO(\F^p)$-invariant. The metric $g_P$ then induces a Riemannian metric on $L$ that makes the projection $P\to L$ a Riemannian submersion. Moreover, the foliation $P\times \F_0$ is a singular Riemannian foliation on $P\times V$ with respect to the product metric $g_P+g_V$, $\OO(\F)$ acts on $P\times V$ by foliated isometries, and the metric projection of $P\times V\to P\times\{0\}=P$, which coincides with the projection on the second factor, is an $\OO(\F)$-equivariant Riemannian submersion.

Taking the quotient by the free $\OO(\F)$ action, there is then an induced metric on  $P\times_{\OO(\F)} V$ such that  $P\times_{\OO(\F)} \F_0$ is a singular Riemannian foliation, and the induced metric projection $P\times_{\OO(\F)} V\to P\times_{\OO(\F)}\{0\}=L$ is a Riemannian submersion.
\end{proof}

\begin{remark}
Given a 4-tuple $(L,V,\F,P)$, by Proposition \ref{P:converse} we can produce a singular Riemannian foliation  $(P\times_{\OO(\F)}V, P\times_{\OO(\F)}\F_0)$ which, by Theorem \ref{MT:slice}, induces a 4 tuple  $(\hat{L},\hat{V},\hat{\F},\hat{P})$. It is not hard to prove that the two 4-tuples are in fact isomorphic, in the sense that there exists a triple $(\phi_L, \phi_V, \phi_P)$ such that $\phi_L:L\to \hat{L}$ is a diffeomorphism, $\phi_V:(V,\F)\to (\hat{V},\hat{\F})$ is a foliated isometry inducing an isomorphism $(\phi_V)_*:\OO(\F)\to \OO(\hat\F)$, and $\phi_P:P\to \hat{P}$ is a $(\phi_V)_*$-equivariant diffeomorphism over $\phi_L$.
\end{remark}

\section{Smooth basic functions}
\label{S:basicfunctions}
The goal of this section is to prove Theorems \ref{MT:sphere} and \ref{MT:manifold}.

Subsection \ref{SS:spheres} concerns Theorem \ref{MT:sphere}. One ingredient in this proof is the continuity of the averaging operator (see Section \ref{S:preliminaries}), which we establish in Lemma \ref{L:averaging}, and which is clear in the homogeneous case. The other ingredient is a result about composite differentiable functions from \cite{BierstoneMilman82} (see also \cite{Tougeron80}).

Subsection \ref{SS:manifolds} presents the proof of Theorem \ref{MT:manifold}. We start with a small digression (Proposition \ref{P:virtuallyabelian}) which serves to point out that the strategy used in \cite{Schwarz75} to prove the homogeneous  case of Theorem \ref{MT:manifold} does not apply in the general (inhomogeneous) case. Indeed, a simple application of \cite[Theorem A]{GGR15} shows that the fundamental groups of leaves need to be virtually Abelian, provided the ambient manifod is compact and simply-connected. Then we prove the existence (Lemma \ref{L:partitionof1})  of smooth partitions of unity for the leaf space, and together with a result from Dimension Theory due to Ostrand, they are used to prove Theorem \ref{MT:manifold}.

\subsection{Basic functions on round spheres}
\label{SS:spheres}
One key fact needed in the proof of the Algebraicity Theorem (see Theorem \ref{T:algebraicity}, or \cite[Theorem 1.1]{LytchakRadeschi15}) is that the averaging operator  takes smooth functions to smooth functions. This is proved via a bootstrapping argument involving elliptic regularity. Using these same tools we prove the following stronger statement, which is needed  in the proof of Theorem \ref{MT:sphere}:
\begin{lemma}[Continuity of averaging]
\label{L:averaging}
Let $(V,\F)$ be an infinitesimal \srf\ with compact leaves. Then the averaging operator $\mathrm{Av}:C^\infty(V)\to C^\infty(V)^\F$ is continuous with respect to the $C^\infty$-topology.
\end{lemma}
\begin{proof}
Recall that if $U\In V$ is an open set whose closure $\bar{U}$ is compact, and $m$ a nonnegative integer, the Sobolev Embedding Theorem (see \cite{Evans}, \S 5.6.3, Theorem 6) implies that the identity map is a continuous inclusion
$$ H^{2m}(U)\to C^{2m-[n/2]-1}(\bar{U})$$
where $H^{2m}(U)=W^{2m,2}(U)$ denotes the Sobolev space of functions on $U$ with square-integrable derivatives up to order $2m$.
Therefore it is enough to show that for every $m$, $U$, and  sequence $\{f_i\}$ of smooth functions on $V$ which converge to zero in the $C^\infty$-topology,  the sequence $\{\mathrm{Av}(f_i)|_U \}$ converges to zero in $H^{2m}(U)$.

We use induction on $m$. For $m=0$, the conclusion follows from the estimate
$$ \|\mathrm{Av}(f_i) \|_{L^2(U)}^2
\leq
\mathrm{vol}(U)( \|\mathrm{Av}(f_i)\|_{C^0(\bar{U})})^2 
\leq
\mathrm{vol}(U)( \|f_i\|_{C^0(\F\cdot\bar{U})})^2 
$$
where $\F\cdot\bar{U}\In V$ denotes the (compact) union of all leaves that meet $\bar{U}$.

Now take $m\geq 1$. We use the fact that averaging commutes with the Laplace operator \cite{LytchakRadeschi15}: $\Delta(\mathrm{Av}(f_i))=\mathrm{Av}(\Delta f_i)$. By elliptic regularity (see \cite{Evans}, \S 6.3.1, Theorem 2), for any relatively compact $U'$ containing $\bar{U}$, we have:
$$ \| \mathrm{Av} (f_i)\|_{H^{2m}(U)}
\leq
C\left( \| \mathrm{Av}(\Delta f_i)\|_{H^{2m-2}(U')} + \|\mathrm{Av}(f_i) \|_{L^2(U')} \right)$$
where $C$ is a constant depending only on $U$ and $U'$.

Applying the inductive hypothesis to $m-1$, $U'$ and the sequence $\{\Delta f_i\}$, we conclude that
$$\| \mathrm{Av}(\Delta f_i)\|_{H^{2m-2}(U')}\to 0.$$
Since $\| \mathrm{Av}(f_i) \|_{L^2(U')}\to 0$ as in the base case, we have  
$\| \mathrm{Av}(f_i)\|_{H^{2m}(U)}\to 0$,
thus finishing the proof.
\end{proof}

\begin{proof}[Proof of Theorem \ref{MT:sphere}]
Since $\R[V]$ is dense in $C^\infty(V)$ (in the $C^\infty$ topology), it follows from Lemma \ref{L:averaging} that $\R[V]^\F$ is dense in $C^\infty(V)^\F$. In particular, $\rho^*(C^\infty(\R^k))\supset\R[V]^\F$ is dense in $C^\infty(V)^\F$. It remains to argue that $\rho^*(C^\infty(\R^k))$ is closed in $C^\infty(V)$.

Since $\rho$ is a polynomial map, its image is semi-algebraic by the Tarski-Seidenberg Theorem, and in particular Nash subanalytic. The map $\rho$ is also proper, because $x\in V \mapsto \|x\|^2$ is a basic polynomial. Therefore we may apply \cite[Theorem 0.2]{BierstoneMilman82}  to conclude that $\rho^*(C^\infty(\R^k))$ is closed in $C^\infty(V)$.
\end{proof}

\subsection{Basic functions on manifolds}
\label{SS:manifolds}
Let $(M,\F)$ be a \srf\  on the complete manifold $M$, and assume the leaves are closed. As usual we say a  function $f$ defined on $M/\F$ (or on an open subset) is smooth if $f\circ\pi$ is smooth, where $\pi:M\to M/\F$ is the natural projection onto the leaf space.

Note that if $M$ is compact, Theorem \ref{MT:sphere} and  Corollary \ref{C:slice} imply that $C^\infty(M)^\F$ is generated, as a $C^\infty$-algebra, by a finite number of basic functions. Theorem \ref{MT:manifold} asserts that the same conclusion holds under the weaker assumption that the  disconnected slice foliations $\F^p$  of $(M,\F)$ fall into a finite set of isomorphism types (see subsections \ref{SS:disconnected} and \ref{SS:localmodels}). Here by an isomorphism between disconnected slice foliations $(V_p,\F^p)$ and $(V_q,\F^q)$ we mean a linear isometry $V_p\to V_q$ which takes leaves of $\F^p$ to leaves of $\F^q$. 

The proof of Theorem \ref{MT:manifold} given below differs from the analogous result in \cite{Schwarz75}: we use Ostrand's Theorem (Theorem 1 in \cite{Ostrand65}) instead of Palais' Theorem about equivariant embeddings into Euclidean space \cite{Palais57}. In fact, Palais' Theorem cannot generalize to \srf s, because not every manifold $L$ is a leaf of some infinitesimal \srf\ $(V,\F)$. Indeed, such a leaf must be contained in a sphere, and we may apply the following consequence of \cite[Theorem A]{GGR15}:
\begin{proposition}[Virtually Abelian]
\label{P:virtuallyabelian}
Let $M$ be a compact, simply-connected Riemannian manifold; $\F$ a \srf\ of $M$ with closed leaves; and $L$ a leaf. Then the fundamental group $\pi_1(L)$ is virtually Abelian.
\end{proposition}
\begin{proof}
Let $L'$ be a nearby regular leaf, and fix some $p\in L$, with slice $V_p=\nu_pL$. By Theorem \ref{MT:slice}, the metric projection $L'\to L$ is a fiber bundle with fiber $L''=L'\cap V_p$. Since $L''$ is compact, the set $\pi_0(L'')$ is finite. Since $\F$ has closed leaves, Theorem A from \cite{GGR15} implies that $\pi_1(L')$ is virtually Abelian. From the long exact sequence of homotopy groups
\begin{equation*}
\ldots\to\pi_1(L')\to\pi_1(L)\to\pi_0(L'')\to\pi_0(L')=0
\end{equation*}
it follows that $\pi_1(L)$ is also virtually Abelian.
\end{proof}

We will need the following lemma in the proof of Theorem \ref{MT:manifold}. Recall that a family of subsets of a topological space is called discrete if every point has a neighbourhood which intersects at most one of the subsets in the family.

\begin{lemma}[Partition of unity]
\label{L:partitionof1}
Let $(M,\F)$ be a \srf\  on the complete manifold $M$, and assume the leaves are closed. 
\begin{enumerate}[a)]
\item There is a partition of unity by smooth functions subordinate to any open covering of the leaf space $M/\F$.
\item Let $\mathcal{U}=\{W_1, W_2, \ldots\}$ be a discrete family of open sets in $M/\F$. Then there exists a smooth function $f$ on $M/\F$ such that $f|_{W_i}$ has constant value $i$.
\end{enumerate}
\end{lemma}
\begin{proof}
\begin{enumerate}[a)]
\item
By Proposition \ref{P:tube}, for each $x\in M/\F$, there is $\epsilon>0$ such that $d(x,\cdot)^2$ is smooth on the ball of radius $\epsilon$ around $x$. Composing this function with appropriate smooth functions on $\R$ yields ``bump'' functions, and hence partitions of unity subordinate to any open cover.
\item
By the definition of discrete family, there is an open cover $\{\mathcal{U}_j\}_{j\in J}$ of $M/\F$ with the property that for every $j\in J$, either there is a unique index $c(j)$ such that $W_{c(j)}\cap\mathcal{U}_j\neq \emptyset$, or $\mathcal{U}_j$ does not intersect any $W_i$. In the latter case, set $c(j)=0$. By part a), there is a partition of unity $\{\phi_j\}$ subordinate to $\{\mathcal{U}_j\}$. Then the smooth function $f=\sum_{j\in J}c(j)\phi_j$ clearly satisfies $f|_{W_i}\equiv i$.
\end{enumerate}
\end{proof}

\begin{proof}[Proof of Theorem \ref{MT:manifold}]
Consider the  cover $\mathcal{C}$ of the leaf space $M/\F$ consisting of the open balls $B(x,r(x))$, for all $x\in M/\F$, where $r(x)$ denotes the focal radius of $\pi^{-1}(x)$ inside $M$. It follows from Proposition \ref{P:tube} that $r(x)>0$. By Theorem 1 in \cite{Ostrand65}, there are at most $1+\dim(M/\F)$ discrete families of open sets whose union covers $M/\F$ and refines $\mathcal{C}$. 

Let $l$ be the number of distinct isomorphism types of slice foliations of $\F$. By partitioning the discrete families produced by Ostrand's Theorem, we obtain discrete families $\mathcal{U}_1, \ldots \mathcal{U}_m$, and disconnected infinitesimal foliations of Euclidean spaces $(V_1,\F_1), \ldots, (V_m,\F_m)$, where $m\leq l(1+\dim(M/\F))$, with the following property: Every open set in $\mathcal{U}_i$ is contained in some ball $B(x,r(x))$ such that the slice foliation of the leaf $\pi^{-1}(x)$ is isomorphic to $(V_i,\F_i)$.

Using a partition of unity of $M/\F$ by smooth functions subordinate to $\{U_i\}$, where  $U_i=\cup_{W\in \mathcal{U}_i}W$, it is enough to show that each $C^\infty(U_i)^\F$ is finitely generated as a $C^\infty$-algebra.

Fix $i$, let $\mathcal{U}_i=\{W_1, W_2, \ldots\}$, so that $U_i=\cup_jW_j$, and $f\in C^\infty(W)^{\F}$ such that $f|_{W_j}\equiv i$. Such $f$ exists by Lemma \ref{L:partitionof1}. Choose generators $\rho_1, \ldots, \rho_k$ for $\R[V_i]^ {\F_i}$.

By Theorem \ref{MT:sphere} and Lemma \ref{L:discosmooth}, $\rho_1, \ldots \rho_k$ generate $C^\infty(V_i)^{\F_i}$, which by Corollary \ref{C:slice} is isomorphic to $C^\infty(W_j)^\F$ for every $j$. Let $\tilde{\rho}_1, \ldots, \tilde{\rho}_k\in C^\infty(U_i)^{\F}$ such that $\tilde{\rho}_m|_{W_j}$ corresponds to $\rho_m$ under this isomorphism.

We claim $\{f,\tilde{\rho}_1, \ldots \tilde{\rho}_k\}$ generate $C^\infty(U_i)^\F$ as a $C^\infty$-algebra. Indeed, given $g\in C^\infty(U_i)^\F$, there are $Q_j\in C^\infty(\R^k)$ such that $g|_{W_j}=Q_j(\tilde{\rho}_1,\ldots \tilde{\rho}_k)$, and hence we may find $Q\in C^\infty(\R^{k+1})$ such that $Q|_{\{j\}\times\R^k}=Q_j$ for $j=1,2,\ldots$, and therefore $g=Q(f,\tilde{\rho}_1,\ldots\tilde{\rho}_k)$.

\end{proof}

Note that, using the notation in the proof above, if $k$ denotes the maximum number of generators for $\R[V_i]^{\F_i}$ over $i=1,\ldots m$, then the total number of generators for $C^\infty(M)^\F$ is bounded above by $l(1+\dim(M/\F))(1+k)$.

\section{Smooth maps between leaf spaces}
\label{S:maps}
The goal of the section is to establish an Inverse Function Theorem for leaf spaces (Theorem \ref{MT:inverse}). We follow closely the results in \cite{Schwarz1980} leading to his Inverse Function Theorem for orbit spaces (Thm. 1.11). Indeed, we start with the definition of tangent spaces and differential of smooth maps, Lemmas \ref{L:tg-space} and  \ref{L:differential}, and in particular relate them to the algebras of basic polynomials in the correponding slices. After using these to prove Theorem \ref{MT:inverse}, we point out (Remark \ref{R:flow}) that the main result in \cite{malex-radeschi-flows-isometries} can be generalized from homogeneous to inhomogeneous \srf s using Theorem \ref{MT:inverse}. In other words, every flow by isometries of the leaf space is smooth.

\subsection{Tangent spaces and differentials}
\label{SS:tangent}
Let $(M,\F)$ be a singular Riemannian foliation with closed leaves, and let $M/\F$ denote the leaf space.

Recall that a map $f:M_1/\F_1\to M_2/\F_2$ between leaf spaces is called \emph{smooth} if $f^*C^\infty(M_2/\F_2)\In C^\infty(M_1/\F_1)$, and it is a \emph{diffeomorphism} is there exists a smooth inverse $f^{-1}:M_1/\F_1\to M_2/\F_2$.

Given a point $x\in M/\F$, let $\mathcal{M}_x$, (or sometimes $\mathcal{M}_x(M/\F)$ if there is a risk of confusion) denote the ideal of germs of smooth functions in $C^\infty(M/\F)$ which vanish at $x$, and define the \emph{tangent space of $M/\F$ at $x$} to be
\[
T_x(M/\F)=(\mathcal{M}_x/\mathcal{M}_x^2)^*,
\]
where $\mathcal{M}_x^2$ denotes the ideal generated by products of pairs of elements in $\mathcal{M}_x$, and  $(\mathcal{M}_x/\mathcal{M}_x^2)^*$ denotes the dual of $\mathcal{M}_x/\mathcal{M}_x^2$. Moreover, given a smooth map $f:M_1/\F_1\to M_2/\F_2$ and a point $x\in M_1/\F_1$, there is an induced map $f^*:\mathcal{M}_{f(x)}\to\mathcal{M}_{x}$ such that $f^*(\mathcal{M}_{f(x)}^2)\In \mathcal{M}_{x}^2$, and therefore there is an induced linear map
\[
d_xf:T_xM_1/\F_1=(\mathcal{M}_x/\mathcal{M}_x^2)^*\to (\mathcal{M}_{f(x)}/\mathcal{M}_{f(x)}^2)^*=T_{f(x)}M_2/\F_2
\]
which we call the \emph{differential of $f$ at $x$}.

If $(M_i,\F_i)$, $i=1,2$, are foliated by points, it is clear that the definitions above coincide with the usual definitions of tangent space and differential.
\begin{lemma}\label{L:tg-space}
Let $(V,\F)$ be a (possibly) disconnected infinitesimal foliation, let $\{\psi_1,\ldots, \psi_l\}$ denote a minimal set of generators for the algebra of basic polynomials, and let $ \hat{\psi}:V/\F\to \R^l$ denote the map induced by $\psi=(\psi_1,\ldots \psi_l):V\to \R^l$. Then
\[
d_0 \hat{\psi}:T_0V/\F\to T_0\R^l\simeq \R^l
\]
is an isomorphism.
\end{lemma}
\begin{proof}
Let $x_1,\ldots,x_l$ denote the standard coordinate functions on $\R^l$. It is enough to prove that
$$ \hat{\psi}^*:\mathcal{M}_0(\R^l)/\mathcal{M}_0^2(\R^l)\to \mathcal{M}_0(V/\F)/\mathcal{M}^2_0(V/\F),\quad [x_i]\to [\psi_i]$$
is an isomorphism. By Theorem \ref{MT:sphere}, $ \hat{\psi}^*\mathcal{M}_0(\R^l)=\mathcal{M}_0(V/\F)$ and this shows surjectivity. On the other hand, $ \hat{\psi}^*$ must be injective otherwise there would be a linear combination $\sum a_i\psi_i=\phi \in \mathcal{M}_0^2(V/\F)$. Without loss of generality, suppose that $a_1\neq 0$. By taking the Taylor series of $\sum a_i\psi_i=\phi$ at $0\in V$, and restricting it to the homogeneous terms of degree $\deg\psi_1$, we would have that $\psi_1$ can be written as a polynomial in $\psi_2,\ldots \psi_l$, which would contradict the minimality of $\psi_1,\ldots, \psi_l$.
\end{proof}

Given a point $x\in M/\F$ and $p\in M$ mapping to $x$ via the projection $M\to M/\F$, by Theorem $C$  a neighbourhood of $x\in M/\F$ is diffeomorphic to $V_p/\F^p$, where $(V_p,\F^p)$ is the disconnected slice foliation at $p$. In particular $T_xM/\F\simeq T_0V_p/\F^p=\R^l$, where $l$ denotes the number of generators of the ring of basic polynomials for $(V_p,\F^p)$.
\begin{lemma}\label{L:differential}
Let $(V_i,\F_i)$, $i=1,2$, be  disconnected infinitesimal foliations, with minimal sets of generators $(\psi_1^i,\ldots \psi_{l_i}^i)$ and induced maps $ \hat{\psi}^i:V_i/\F_i\to \R^{l_i}$. If $f:V_1/\F_1\to V_2/\F_2$ is a smooth map, then there is a smooth map $\phi:\R^{l_1}\to \R^{l_2}$ such that the following diagram commutes:
\begin{equation}\label{E:square-differential}
\begin{CD}
V_1/\F_1    @>f>>  V_2/\F_2 \\
@VV\hat{\psi}^1 V        @VV\hat{\psi}^2V \\
\R^{l_1}    @>\phi>>  \R^{l_2} 
\end{CD}
\end{equation}
\end{lemma}

\begin{proof}
Since $ \hat{\psi}^2\circ f:V_1/\F_1\to \R^{l_2}$ is smooth, in particular $ \hat{\psi}^2_i\circ f\in C^{\infty}(V_1/\F_1)$ for every $i=1,\ldots, l_2$ and, by Theorem B, there exists some $\phi_i\in C^{\infty}(\R^{l_1})$ such that $ \hat{\psi}^2_i\circ f=\phi_i\circ \hat{\psi}^1$. By construction, the function $\phi=(\phi_1,\ldots \phi_{l_2})$ then makes the diagram commute.
\end{proof}

It follows from Lemmas \ref{L:tg-space} and \ref{L:differential} that the differential of a smooth map $f:V_1/\F_1\to V_2/\F_2$ between leaf spaces of infinitesimal foliations can be identified with the differential of a smooth map $\phi:\R^{l_1}\to \R^{l_2}$ in the usual sense.

%

\subsection{Inverse Function Theorem}
\label{SS:inverse}
Finally we turn to the proof of the Inverse Function Theorem, which follows closely the proof of \cite[Theorem 1.11]{Schwarz1980}. Recall that given  a \srf\  $(M,\F)$, $M$ has a natural stratification by dimension of leaves. The top stratum, consisting of leaves of maximal dimension, is the complement of a closed set of codimension at least two, and is therefore open, dense and connected \cite[page 197]{Molino}.
\begin{proof}[Proof of Theorem \ref{MT:inverse}]
By the Slice Theorem, there are neighbourhoods of $x\in M_1/\F_1$ and $f(x)\in M_2/\F_2$ diffeomorphic respectively to leaf spaces $V_1/\F_1$ and $V_2/\F_2$ of (possibly disconnected) infinitesimal foliations. Therefore $f$ restricts to a smooth strata-preserving map $f:V_1/\F_1\to V_2/\F_2$ and, by Lemma \ref{L:differential}, there is an induced smooth map $\phi:\R^{l_1}\to \R^{l_2}$ which makes the diagram in \eqref{E:square-differential} commute. Because now $d_0f$ is an isomorphism, by Lemma \ref{L:tg-space} the differential $d_0\phi$ is an isomorphism as well, and by the standard Inverse Function Theorem $\phi$ is a diffeomorphism near $0\in \R^{l_1}$.

Let $U_i\In \R^{l_i}$, $i=1,2$, be open neighbourhoods of $0$ such that $\phi|_{U_1}:U_1\to U_2$ is a diffeomorphism, and let $S_i\In U_i$ denote the image of $V_i/\F_i$, with $S_i^{pr}$ denoting the image of the principal part and $S_i^s=S_i\setminus S_i^{pr}$ denoting the image of the singular part. Because $f$ is strata-preserving, $f(S_1^{pr})\In S_2^{pr}$ and $f(S_1^{s})\In S_2^{s}$. In particular, $f(S_1^{pr})$ is open and closed in $S_2^{pr}$ and, by choosing $U_2$ so that $S_2^{pr}$ is connected (for example taking $U_2$ a distance ball around the origin) it follows that $f(S_1^{pr})=S_2^{pr}$. Because $S_i$ is the closure of $S_i^{pr}$ in $U_i$, $i=1,2$, it follows that $f(S_1)=S_2$ as well. Therefore, it makes sense to define
\[
f^{-1}=(\hat{\psi}^1)^{-1}\circ (\phi|_{S_1})^{-1}\circ \hat{\psi}^2:V_2/\F_2\to V_1/\F_1
\]
which is smooth because it induces the map $\phi^{-1}$ between the smooth functions.
\end{proof}

\begin{remark}
\label{R:flow}
The Inverse Function Theorem above could be applied for example to prove that, given a complete manifold $M$ and a singular Riemannian foliation $(M,\F)$ with closed leaves, any one-parameter group of isometries
\[
\phi: M/\F\times \R\to M/\F
\]
is smooth. Indeed, this was proved in \cite{malex-radeschi-flows-isometries} in the homogeneous setting. The proof is divided in three major steps (reducing the proof to a simpler situation; proving the smoothness on a codimension-one set of $M/\F\times \R$; and finally extending the smoothness to the whole of $M/\F\times\R$), where the first two steps hold for general singular Riemannian foliations, and the third is a simple application of Schwarz's Inverse function Theorem for orbit spaces \cite[Thm. 1.11]{Schwarz1980}. As Theorem  \ref{MT:inverse} generalizes Schwarz's  Inverse function Theorem to the case of leaf spaces, this could be applied to extend the main result in \cite{malex-radeschi-flows-isometries} to general singular Riemannian foliations with closed leaves.
\end{remark}

\bibliographystyle{alpha}
\bibliography{refsmooth,ref}
\end{document}